\documentclass[11pt]{article}\usepackage{fullpage,amsmath,amsfonts,amsthm,graphicx,amssymb,textcomp,enumerate,multicol,rotating,stmaryrd}
\usepackage{hyperref}

\newtheorem{theorem}{Theorem}[section]
\newtheorem{corollary}[theorem]{Corollary}
\newtheorem{lemma}[theorem]{Lemma}
\newtheorem{proposition}[theorem]{Proposition}

\theoremstyle{definition}
\newtheorem{definition}[theorem]{Definition}

\theoremstyle{remark}

\theoremstyle{remark}
\newtheorem{remark}[theorem]{Remark}

\theoremstyle{remark}

\newcommand{\C}{\mathbb{C}}

\newcommand{\Z}{\mathbb{Z}}  

\newcommand{\B}{\mathcal{B}}
\newcommand{\T}{\mathcal{T}}
\newcommand{\X}{\mathcal{X}}
\newcommand{\PP}{\mathbf{P}}
\newcommand{\TLn}{\mathbb{TL}_n}
\newcommand{\D}{\mathcal{D}}

\newcommand{\cone}{\text{Cone}}
\newcommand{\h}{\mathrm{h}}
\newcommand{\q}{\mathrm{q}}
\newcommand{\hocolim}{\text{hocolim}}
\newcommand{\la}{\left\langle}
\newcommand{\ra}{\right\rangle}

\newcommand{\KC}{\h^{n^-}\q^{-N}KC}
\newcommand{\KCq}[1]{\h^{n^-}\q^{-N+#1}KC}
\newcommand{\KChq}[2]{\h^{n^-+#1}\q^{-N+#2}KC}
\newcommand{\XX}{\Sigma^{n^-}\q^{-N}\X}

\newcommand{\vres}{\includegraphics[scale=0.21]{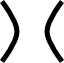}}
\newcommand{\hres}{\includegraphics[scale=0.21]{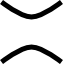}}
\newcommand{\crossing}{\includegraphics[scale=0.21]{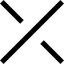}}
\newcommand{\lcrossing}{\includegraphics[scale=0.21]{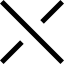}}
\newcommand{\poscross}{\includegraphics[scale=0.21]{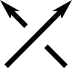}}
\newcommand{\negcross}{\includegraphics[scale=0.21]{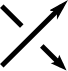}}
\newcommand{\inlinegfx}[1]{\includegraphics[scale=0.21]{#1}\hspace{.01in}}

\title{The Khovanov homology of infinite braids}
\author{Gabriel Islambouli (\href{mailto:gfi8ps@virginia.edu}{\texttt{gfi8ps@virginia.edu}})\\
Michael Willis (\href{mailto:msw3ka@virginia.edu}{\texttt{msw3ka@virginia.edu}})\\
Department of Mathematics, University of Virginia
}

\begin{document}

\maketitle

\begin{abstract}
We show that the limiting Khovanov chain complex of any infinite positive braid categorifies the Jones-Wenzl projector.  This result extends Lev Rozansky's categorification of the Jones-Wenzl projectors using the limiting complex of infinite torus braids.  We also show a similar result for the limiting Lipshitz-Sarkar-Khovanov homotopy types of the closures of such braids.  Extensions to more general infinite braids are also considered.
\end{abstract}

\section{Introduction}
\label{sec_intro}
In the Temperley-Lieb algebra $TL_n$ on $n$ strands over the fraction field $\C(q)$, there is a special idempotent element $P_n$, the $n$-strand Jones-Wenzl projector \cite{Wenzl}.  Such projectors have been studied extensively, and are used in the construction of various 3-manifold and spin network invariants \cite{KL}.

In the Bar-Natan categorification of $TL_n$ \cite{BN}, the categorified projector can be represented by a semi-infinite chain complex $\PP_n$ of Temperley-Lieb diagrams and maps (cobordisms) between them (by semi-infinite we mean a complex with homological degree bounded below but unbounded above).  The graded Euler characteristic of this complex recovers a power series expansion in the variable $q$ (or $q^{-1}$) of the rational terms in the original $P_n$.  Ben Cooper and Slava Krushkal constructed such $\PP_n$ inductively in \cite{CK}.  At roughly the same time, in \cite{Roz}, Lev Rozansky showed that such a $\PP_n$ could be constructed as the (properly normalized) limiting Khovanov chain complex $KC(\T^\infty)$ (taken in the sense of \cite{BN}; see Section \ref{sec_Kauff and Khov}) associated to the infinite torus braid $\T^\infty$ on $n$ strands.  Universality properties described in \cite{CK} ensure that the two constructions must be chain homotopy equivalent.

The main goal of the paper is to prove the following theorem, showing that the categorified projector $\PP_n$ may be obtained using essentially any infinite positive braid in the place of the infinite torus braid.
\begin{theorem}
\label{inf braid gives jw}
Let $\B$ be any complete semi-infinite positive braid, viewed as the limit of positive braid words
\[ \B = \lim_{\ell\rightarrow\infty} \sigma_{j_1}\sigma_{j_2}\cdots\sigma_{j_\ell}.\]
Then the limiting Khovanov chain complex $KC(\B)$ satisfies
\begin{equation}
\label{inf braid gives jw eqn}
KC(\B) := \lim_{\ell\rightarrow\infty} \h^a\q^b KC( \sigma_{j_1}\sigma_{j_2}\cdots\sigma_{j_\ell} ) \simeq \PP_n
\end{equation}
where $\h^a$ and $\q^b$ denote homological and $q$-degree shifts, respectively.  In other words, $KC(\B)$ for any such $\B$ is chain homotopy equivalent to the categorified projector $\PP_n$.
\end{theorem}
We will clarify the notion of completeness, as well as the grading shifts $a$ and $b$, in Section \ref{sec_proof}.  Here we quickly note the following simple corollary of Theorem \ref{inf braid gives jw}.
\begin{corollary}
\label{inf braid with neg Kauffman bracket}
The (suitably normalized) Kauffman bracket of any complete semi-infinite positive braid $\B$ stabilizes to give a power series representation of the Jones-Wenzl projector $P_n$.
\end{corollary}

In \cite{LS}, Robert Lipshitz and Sucharit Sarkar defined a homotopy type invariant $\X(L)$ of a link $L$, which we shall refer to as the Lipshitz-Sarkar-Khovanov (abbreviated as L-S-K) homotopy type of $L$.  $\X(L)$ is a suspension spectrum of a CW complex with cellular cochain complex satisfying $C^*(\X(L)) \simeq KC^*(L)$.  In \cite{MW,MW2}, one of the authors showed that the homotopy types of closures of infinite twists also have a well-defined limit with cochain complex recovering the corresponding closure of Rozansky's $\PP_n$ (this result was independently proven in \cite{LOS}).  Abusing the notation for $q$-degree shifts, we have the following theorem similar to Theorem \ref{inf braid gives jw} above:
\begin{theorem}
\label{inf braid gives jw X}
Let $\overline{\B}$ denote any closure of a complete semi-infinite positive braid $\B$ as in Theorem \ref{inf braid gives jw}.  Let $\overline{\T^\infty}$ denote the corresponding closure of the infinite twist.  Then
\begin{equation}
\label{inf braid gives jw X eqn}
\X(\overline{\B}):= \lim_{\ell\rightarrow\infty} \Sigma^a\q^b \X( \overline{\sigma_{j_1}\sigma_{j_2}\cdots\sigma_{j_\ell}} ) \simeq \X(\overline{\T^\infty})
\end{equation}
where again $a$ and $b$ stand for homological shifts (via suspensions $\Sigma$) and $q$-degree shifts.
\end{theorem}
Figure \ref{closures of infinite braids} illustrates a closure of $\B$ and the corresponding closure of $\T^\infty$.  Note that as of this writing, the L-S-K homotopy type for braids and/or tangles has not yet been defined; as such, Theorem \ref{inf braid gives jw X} is the closest notion available to a lifting of Theorem \ref{inf braid gives jw} to the stable homotopy category.

\begin{figure}
\begin{center}
\includegraphics[scale=.5]{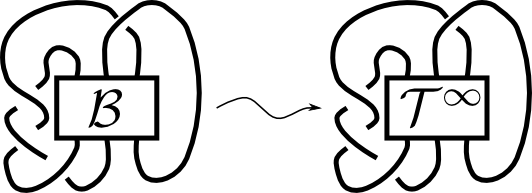}
\end{center}
\caption{A possible closure $\overline{\B}$ of the infinite braid $\B$, and the corresponding closure $\overline{\T^\infty}$ of the infinite twist.}
\label{closures of infinite braids}
\end{figure}

The proofs of Theorems \ref{inf braid gives jw} and \ref{inf braid gives jw X} are essentially the same argument.  In short, we show that $KC(\B)$ contains a copy of $KC(\T^\infty)$ along with some error terms which are pushed out further and further so that in the limit we see only $KC(\T^\infty)$ (and similarly for $\X(B)$).  In slightly more detail, the assumption that $\B$ is complete will ensure that $\B$ `contains' the crossings that would make up the infinite twist $\T^\infty$, as well as potentially many other `extra' positive crossings.  If we resolve all of the `extra' crossings as 0-resolutions, we see $\T^\infty$.  If we resolve some of the `extra' crossings as 1-resolutions, we see mixtures of twists and turnbacks, which allow for simplifications via pulling the turnbacks through the twists.  Careful tracking of the homological degrees during this process will show that, in the limit, $KC^i(\B)$ will match $KC^i(\T^\infty)$ for any $i$, while careful tracking of the $q$-degrees will achieve a similar result for $\X(\B)$.

Although the statements of Theorems \ref{inf braid gives jw} and \ref{inf braid gives jw X} are for complete semi-infinite positive braids, the results together with properties of the Jones-Wenzl projectors (and the corresponding homotopy types) quickly lead to several corollaries involving more general notions of infinite braids.  Roughly speaking, any tangle diagram that contains positive infinite braids has Khovanov homology and homotopy type (for links) matching that of the same diagram with infinite twists replacing the infinite braids.  More precise statements along these lines can be found in the final section of the paper.

This paper is arranged as follows.  In Section \ref{sec_bkgrnd} we review the relevant background needed for the Khovanov homology of the infinite braids, as well as recalling Rozansky's results on the infinite twist.  In Section \ref{sec_proof} we give a detailed proof of Theorem \ref{inf braid gives jw}.  In Section \ref{sec_proof X} we prove Theorem \ref{inf braid gives jw X}, highlighting the slight differences between this and the first proof.  Finally in Section \ref{sec_general inf braids} we explore corollaries of these theorems that give statements about more general infinite braids.

\subsection{Acknowledgements}
\label{sec_ack}

The authors would like to thank their advisor Slava Krushkal for posing the original question about the Kauffman bracket of iterated braids, as well as for many helpful discussions in the early stages of pursuing the result.

\section{Background and Conventions}
\label{sec_bkgrnd}

\subsection{The (Categorified) Temperley-Lieb Algebra and the Jones-Wenzl Projectors}
\label{sec_TL and jw}

In this paper, $TL_n$ will denote the Temperley-Lieb algebra on $n$ strands over the field $\C(q)$, where $q$ is a formal variable.  For a full introduction to this algebra and some of its uses in 3-manifold theory, see \cite{KL}.  Here we simply recall that $TL_n$ is generated by planar diagrams of crossingless $(n,n)$ tangles which we shall draw vertically.  Multiplication is defined by (downward) concatenation, and there is the local relation that a circle can be deleted with the resulting diagram scaled by the factor $(q+q^{-1})$.  The multiplicative identity is the diagram of $n$ vertical lines, and will be denoted by $I_n$.  It is well-known that $TL_n$ is multiplicatively generated by the diagrams $\{e_i|i=1,\dots,n-1\}$ described in Figure \ref{ei diagrams}.

\begin{figure}
\begin{center}
\includegraphics[scale=.5]{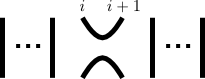}
\end{center}
\caption{The diagrams $e_i\in TL_n$ that form the standard multiplicative basis.}
\label{ei diagrams}
\end{figure}

We shall use the notation $\TLn$ to denote the graphical categorification of $TL_n$ of Dror Bar-Natan in \cite{BN}.  An excellent summary of this construction is provided in section 2.3 of \cite{CK}.  The objects are chain complexes of (direct sums of) $q$-graded Temperley-Lieb diagrams, with differentials based on `dotted' cobordisms between such diagrams modulo some local relations that allow for (among other things) an isomorphism between the circle and the direct sum $q^{-1}(\emptyset)\oplus q(\emptyset)$.  The exact nature of these maps and relations will not be relevant for the arguments in this paper.

Within $TL_n$ there is a special idempotent element $P_n$ characterized by the following axioms:
\begin{enumerate}[I.]
\item $P_n\cdot e_i = e_i\cdot P_n = 0$ for any of the standard multiplicative generators $e_i\in TL_n$.  This is often described by stating that $P_n$ is ``killed by turnbacks''.
\item The coefficient of the $n$-strand identity tangle $I_n$ in the expression for $P_n$ is 1.
\end{enumerate}
These are the Jones-Wenzl projectors, originally defined in \cite{Wenzl}.  The simplest non-trivial example is $P_2$, shown below.
\begin{align}
P_2 &= I_2 - \frac{1}{q^{-1}+q}\la\hres\ra\\
&= I_2 + (-q+q^3-q^5+\cdots)\la\hres\ra.
\label{P2 equation}
\end{align}

In the categorified world of $\TLn$ there is also a special semi-infinite chain complex $\PP_n$, characterized up to chain homotopy equivalence by the similar axioms:
\begin{enumerate}[I.]
\item $\PP_n\otimes \mathbf{e}_i \simeq \{*\}$ for any $TL_n$ generator $e_i$ viewed as a one-term complex in $\TLn$.  That is, $\PP_n$ is ``contractible under turnbacks''.
\item The identity diagram $I_n$ appears only once, in homological degree zero.
\item All negative homological degrees and $q$-degrees of $\PP_n$ are empty, and the differentials are made up of degree zero maps.
\end{enumerate}
Such a complex $\PP_n$ is called a \emph{categorified Jones-Wenzl projector}.  For more details on this axiomatic definition, see \cite{CK} where such complexes are constructed inductively.  The simplest non-trivial example is $\PP_2$, shown below (compare to Equation (\ref{P2 equation})).
\begin{equation}
\label{PP2 equation}
\PP_2 = \vres \longrightarrow q\hspace{.06in}\hres \longrightarrow q^3\hspace{.06in}\hres \longrightarrow q^5\hspace{.06in}\hres\longrightarrow\cdots
\end{equation}
The maps in the complex (\ref{PP2 equation}) are given explicitly in \cite{CK}.  Note that the graded Euler characteristic for $\PP_2$ gives precisely the power series representation of $P_2$ from Equation (\ref{P2 equation}).  The infinite complex is necessitated by the lack of a straightforward notion of categorifying a rational function of $q$, leading to the use of the corresponding power series instead.

\begin{remark}
\label{negative power series remark}
The version of $\PP_n$ described above is based upon expanding the ratios in $P_n$ as power series in the variable $q$.  However, it is equally valid to expand them as power series in the variable $q^{-1}$.  Thus the third axiom of $\PP_n$ could be replaced by a similar axiom declaring the \emph{positive} homological and $q$-gradings to be empty.  We shall focus on the $q$-expansion in this paper, leading to the statements about infinite positive braids; however, the same story could be told focusing on the $q^{-1}$-expansion, leading to equivalent statements about infinite negative braids.  See also Remark \ref{Roz left handed remark}.
\end{remark}

\subsection{The Kauffman Bracket and Khovanov Chain Complex of a Tangle}
\label{sec_Kauff and Khov}

The Temperley-Lieb algebra $TL_n$ is related to $(n,n)$-tangles via the Kauffman bracket $\la \cdot \ra$, a function converting $(n,n)$-tangles into elements of $TL_n$ (see \cite{KL}).  The categorified version of this is the Khovanov chain complex $KC(\cdot)$ which associates to an $(n,n)$-tangle a chain complex in $\TLn$, whose graded Euler characteristic returns the Kauffman bracket of the tangle (for the original definition of Khovanov homology, see \cite{Khov}, with the extension to tangles in \cite{Khov2}; we follow the framework for tangles in \cite{BN} where the functor $KC(\cdot)$ is referred to as the \emph{formal Khovanov Bracket}, denoted $\llbracket \cdot \rrbracket$).  There exist several different normalization conventions for both the Kauffman bracket and the Khovanov chain complex in the literature.  In the hopes of keeping some consistency with one author's earlier work, we adopt the following conventions.

\begin{gather}
\la \poscross \ra = q\la \vres \ra - q^2 \la \hres \ra\\
\label{negcross KC}
\la \negcross \ra = -q^{-2}\la \vres \ra + q^{-1} \la \hres \ra\\
KC\left( \poscross \right) = \q^1\left( \vres \longrightarrow \q^1\text{ }\hres \right)
\label{KC poscros}\\
KC\left( \negcross \right) = \h^{-1}\q^{-2} \left( \vres \longrightarrow \q^{-1}\text{ }\hres \right)
\label{KC negcros}
\end{gather}

Here we use the symbols $\h$ and $\q$ to indicate homological and $q$-degree shifts respectively, as in \cite{Roz}.  The maps in the Khovanov chain complexes are saddle cobordisms between the two resolutions, and the vertical resolutions are placed in homological degree zero (so that after the $\h^{-1}$ shift for the negative crossing \negcross, it is the horizontal resolution that is in homological degree zero).  Under this convention, both the Kauffman bracket and the Khovanov chain complex are true invariants of tangles.  That is to say, they are invariant under all three Reidemeister moves, without the need for any grading shifts.  The tradeoff for this seemingly natural choice is that many of the formulae required for manipulating the Khovanov complexes of braids will require various shifts of both homological and $q$-degrees, as we shall see in the following sections.

\begin{remark}
\label{handedness vs sign}
Moving forward, we will be referring to crossings (\crossing) as \emph{right-handed} rather than positive, and braids are then called \emph{right-handed} if every crossing within them is right-handed (the usual definition of a positive braid).  As such, the braids of Theorems \ref{inf braid gives jw} and \ref{inf braid gives jw X} will henceforth be referred to as right-handed rather than positive.  Note that, in the case of a braid viewed as an $(n,n)$-tangle with all of the strands oriented upward, right-handedness and positivity of crossings are equivalent.  However, if the strands of a braid are oriented in different directions (as may be required if the braid is closed in the 3-sphere by a turnback, say), we could see right-handed crossings that are actually negative, as in Equation (\ref{negcross KC}) above.
\end{remark}

The crossing rules (\ref{KC poscros}) and (\ref{KC negcros}) allow us to view the Khovanov chain complex of a tangle as a mapping cone in the usual way.  Using our normalization conventions, the relevant statement is as follows.

\begin{lemma}
\label{main mapping cone}
Let $\mathbf{T}$ be an oriented tangle, with a specified crossing \crossing.  Let $\mathbf{T}_0$ denote the same tangle with the crossing replaced by its 0-resolution \vres , and let $\mathbf{T}_1$ denote the same with the 1-resolution \hres .  Then the shifted Khovanov complex of $\mathbf{T}$ can be viewed as a mapping cone:
\begin{equation}
\label{main mapping cone eqn extra notation}
\h^{n^-}\q^{-N} KC(\mathbf{T}) = \cone\left( \h^{n^-_0}\q^{-N_0} KC(\mathbf{T}_0) \longrightarrow \h^{n^-_1}\q^{-N_1+1} KC(\mathbf{T}_1) \right).
\end{equation}
Here $n^-$ indicates the number of negative crossings in $\mathbf{T}$, while $N$ indicates $n^+-2n^-$, the number of positive crossings minus twice the number of negative crossings in $\mathbf{T}$.  The subscripts $n^-_i$ and $N_i$ indicate the same counts of crossings in $\mathbf{T}_i$ for $i=0,1$.
\end{lemma}
The main arguments of this paper will use this construction in an iterated fashion over many crossings.  As such, the $\cone()$ notation and the subscripts for the grading shifts quickly become unwieldy.  For this reason, we drop the word $\cone$ from the notation and adopt the following convention:

\begin{definition}
\label{norm notation}
The symbols $n^-$ and $N:=n^+-2n^-$ will count positive and negative crossings in whatever tangle they appear with.  Thus Equation (\ref{main mapping cone eqn extra notation}) will be written as
\begin{equation}
\label{main mapping cone eqn}
\h^{n^-}\q^{-N} KC(\mathbf{T}) = \left( \h^{n^-}\q^{-N} KC(\mathbf{T}_0) \longrightarrow \h^{n^-}\q^{-N+1} KC(\mathbf{T}_1) \right)
\end{equation}
and it will be understood that the various $n^-$ and $N$ are actually different numbers within this mapping cone.
\end{definition}

\begin{corollary}
\label{mapping cone proj}
Given tangles $\mathbf{T},\mathbf{T_0}$, and $\mathbf{T_1}$ as in Lemma \ref{main mapping cone}, there is a chain map $\KC(\mathbf{T})\rightarrow\KC(\mathbf{T_0})$ with mapping cone that is chain homotopy equivalent to $\h^{n^-+1}\q^{-N+1}KC(\mathbf{T_1})$.
\end{corollary}

Now in our normalization, $KC$ itself is invariant under all Reidemeister moves.  Combining this with the notational convention of Definition \ref{norm notation} gives the following shifts for the (negative) Reidemeister I and Reidemeister II moves that we shall need later.

\begin{gather}
\label{R1 lower}
\h^{n^-}\q^{-N} KC\left( \inlinegfx{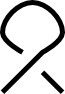} \right) \simeq \h^{n^-+1}\q^{-N+2} KC\left( \inlinegfx{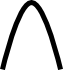} \right)\\
\label{R2}
\h^{n^-}\q^{-N} KC\left( \inlinegfx{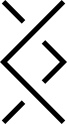} \right) \simeq \h^{n^-+1}\q^{-N+1} KC\left( \inlinegfx{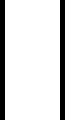} \right) \simeq \h^{n^-}\q^{-N} KC\left( \inlinegfx{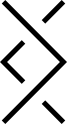} \right)
\end{gather}
Compare these shifts to those that occur using the grading conventions in \cite{Roz}.  Meanwhile, since Reidemeister III moves only change the arrangement of crossings rather than their number or orientation, we see that Reidemeister III moves incur no shifts to either homological or $q$-grading even within this renormalized setting.

\subsection{The Infinite Twist as Categorified Projector}
\label{sec_inf twist as jw}

\begin{definition}
In the braid group $\mathfrak{B}_n$ on $n$ strands, the symbol $\T$ will denote the \emph{fractional (right-handed) twist} $\T:=\sigma_1\sigma_2\cdots\sigma_{n-1}$.  The \emph{full (right-handed) twist} is then the braid $\T^n$.  See Figure \ref{twist and full twist} for clarification.
\end{definition}

\begin{figure}
\begin{center}
\includegraphics[scale=.28]{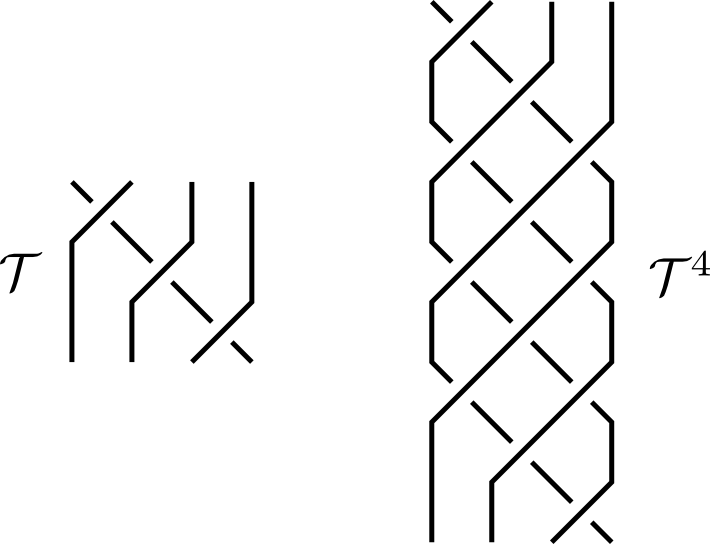}
\end{center}
\caption{The fractional twist $\T$ and the full twist $\T^n$ in the case $n=4$.}
\label{twist and full twist}
\end{figure}

In \cite{Roz} Lev Rozansky provided a notion of a system of chain complexes stabilizing to some limiting complex, and proved the following theorem.

\begin{theorem}[Theorem 2.2 in \cite{Roz}]
\label{Roz thm}
The shifted Khovanov chain complexes $\h^{n^-}\q^{-N}KC(\T^{kn})$ stabilize to a limiting complex
\begin{equation}
\label{Roz full twist eqn}
KC(\T^\infty) := \lim_{k\rightarrow\infty} \h^{n^-}\q^{-N}KC(\T^{kn})
\end{equation}
which satisfies the axioms of a categorified projector $\PP_n$.
\end{theorem}
\begin{remark}
\label{Roz left handed remark}
In fact Rozansky's original result concerned left-handed rather than right-handed twisting, but the methods clearly translate to the right handed case with no trouble.  The left-handed version recovers a power series expansion of $P_n$ in the variable $q^{-1}$; see Remark \ref{negative power series remark}.
\end{remark}
For a full account of the notions involved with such limiting complexes, see \cite{Roz}.  Here we recall only the material most helpful for our current purposes, translated to right-handed twisting.
\begin{definition}
\label{max hom deg of equivalence}
Given a chain map $\mathbf{A}\xrightarrow{f}\mathbf{B}$ between chain complexes, let $|f|_h$ denote the maximal degree $d$ for which the complex $\cone(f)$ is chain homotopy equivalent to a complex $\mathbf{C}$ that is trivial below homological degree $d$.
\end{definition}
In essence, $|f|_h$ denotes the maximal homological degree through which the map $f$ gives a chain homotopy equivalence between $\mathbf{A}$ and $\mathbf{B}$.  Note that all of the chain complexes being discussed here have differential \emph{increasing} homological degree by 1 (as in the Khovanov chain complex).
\begin{definition}
\label{inverse system def}
An \emph{inverse system} of chain complexes is a sequence of chain maps
\begin{equation}
\label{inverse system eqn}
\{\mathbf{A_k},f_k\}:=\mathbf{A_1} \xleftarrow{f_1} \mathbf{A_2} \xleftarrow{f_2} \cdots
\end{equation}
Such a system is called \emph{Cauchy} if the maps $f_k$ satisfy $|f_k|_h\rightarrow\infty$ as $k\rightarrow\infty$.
\end{definition}
\begin{definition}
An inverse system $\{\mathbf{A_k},f_k\}$ has a \emph{(inverse) limit} $\mathbf{A_\infty}:=\lim_{k\rightarrow\infty} \mathbf{A_k}$ if there exist maps $\tilde{f}_k:\mathbf{A_\infty}\rightarrow\mathbf{A_k}$ that commute with the system maps $f_k$ such that $|\tilde{f}_k|_h\rightarrow\infty$ as $k\rightarrow\infty$.
\end{definition}
\begin{theorem}[Theorem 2.5 in \cite{Roz}]
\label{Cauchy has limit}
An inverse system of chain complexes $\{\mathbf{A_k},f_k\}$ has a limit $\mathbf{A_\infty}$ if and only if it is Cauchy.
\end{theorem}
Unwinding the definitions and results in \cite{Roz}, we see that the limiting complex $\mathbf{A_\infty}$ of Theorem \ref{Cauchy has limit} is, up through homological degree $d$, chain homotopy equivalent to the corresponding $\mathbf{A_{k_0}}$ beyond which all of the maps $f_{k\geq k_0}$ satisfy $|f_{k\geq k_0}|_h\geq d$.  In this sense the chain complexes $\mathbf{A_k}$ stabilize to give the limiting complex $\mathbf{A_\infty}$ ``one homological degree at a time''.  Thus if we have a second inverse system of $\mathbf{B_\ell}$'s with homotopy equivalences to the $\mathbf{A_k}$'s up through ever-increasing homological degrees, we should be able to conclude that $\mathbf{B_\infty} \simeq \mathbf{A_\infty}$.  The following proposition clarifies this idea.
\begin{proposition}
\label{comparing systems}
Suppose $\{\mathbf{A_k},f_k\}$ and $\{\mathbf{B_\ell},g_\ell\}$ are Cauchy inverse systems with limits $\mathbf{A_\infty}=\lim_{k\rightarrow\infty} \mathbf{A_k}$ and $\mathbf{B_\infty}=\lim_{\ell\rightarrow\infty} \mathbf{B_\ell}$ respectively.  Suppose there are maps
\[F_\ell : \mathbf{B_\ell} \rightarrow \mathbf{A_{k=z(\ell)}}\]
($z(\ell)$ is an increasing function of $\ell$, not necessarily strict) forming a commuting diagram with the system maps $f_k$ and $g_\ell$.  If $|F_\ell|_h\rightarrow\infty$ as $\ell\rightarrow\infty$, then $\mathbf{B_\infty}\simeq\mathbf{A_\infty}$.
\end{proposition}
\begin{proof}
Similar to the proof of Proposition 3.13 in \cite{Roz}, the definition of the limit provides maps that compose with the maps $F_\ell$ to give maps $\mathbf{B_\infty}$ to all of the $\mathbf{A_k}$, and thus there is a map $F_\infty:\mathbf{B_\infty}\rightarrow\mathbf{A_\infty}$ making commutative diagrams with all of the $\tilde{f}_k$ (see Theorem 3.9 in \cite{Roz}).  All of the other maps have homological order going to infinity as $\ell$ and $k$ go to infinity, forcing $|F_\infty|_h=\infty$ and thus $\mathbf{B_\infty}\simeq\mathbf{A_\infty}$.  Figure \ref{comparing systems fig} illustrates the situation that will occur within this paper.
\end{proof}
\begin{figure}
\begin{center}
\includegraphics[scale=.5]{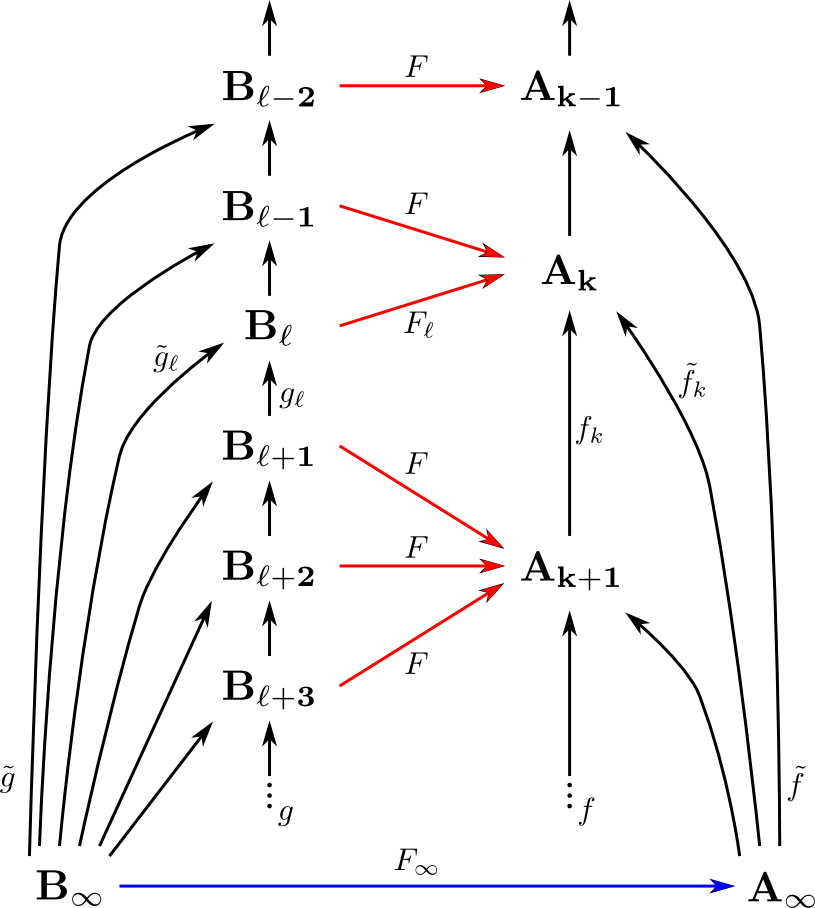}
\end{center}
\caption{The diagram for Proposition \ref{comparing systems}.  Given the two Cauchy systems $\{\mathbf{A_k},f_k\}$ and $\{\mathbf{B_\ell},g_\ell\}$, \cite{Roz} provides the complexes $\mathbf{A_\infty},\mathbf{B_\infty}$ and the maps $\tilde{f},\tilde{g}$.  If we can find maps $F$ (shown in red), then \cite{Roz} also provides the map $F_\infty$ (blue).  If we can show $|F_\ell|_h\rightarrow\infty$ as $\ell\rightarrow\infty$, then $F_\infty$ is a chain homotopy equivalence.}
\label{comparing systems fig}
\end{figure}

\section{Proving Theorem \ref{inf braid gives jw}}
\label{sec_proof}

\subsection{An Overview}
\label{sec_overview}

\begin{definition}
\label{inf pos braid def}
A \emph{semi-infinite right-handed braid} $\B$ on $n$ strands is a semi-infinite word in the standard generators $\sigma_i$ of the braid group $\mathfrak{B}_n$
\begin{equation}
\label{inf pos braid eqn}
\B:=\sigma_{j_1}\sigma_{j_2}\cdots
\end{equation}
Such a braid is called \emph{complete} if each $\sigma_i$ for $i=1,2,\dots,n-1$ occurs infinitely often in the word for $\B$.
\end{definition}
Such an infinite braid $\B$ is called right-handed because there are no left-handed crossings ($\sigma_i^{-1}$) allowed.

\begin{definition}
\label{partial braid def}
Given a semi-infinite right-handed braid $\B=\sigma_{j_1}\sigma_{j_2}\cdots$, the \emph{$\ell^{\text{th}}$ partial braid of $\B$} shall be the braid $\B_{\ell}:=\sigma_{j_1}\sigma_{j_2}\cdots\sigma{j_\ell}$.
\end{definition}

The proof of Theorem \ref{inf braid gives jw} will be based upon Proposition \ref{comparing systems} and, in particular, Figure \ref{comparing systems fig}.  With that diagram in mind, we have the following correspondences.

\begin{enumerate}
\item The chain complexes $\KC(\T^{kn})$ will play the role of the $\mathbf{A_k}$.
\item Theorem \ref{Roz thm} then guarantees that $\mathbf{A_\infty}\simeq\PP_n$.
\item Given a semi-infinite right-handed braid $\B:=\sigma_{j_1}\sigma_{j_2}\cdots$, the chain complexes $\KC(\B_{\ell})$ will play the role of the $\mathbf{B_\ell}$.
\item Each map $g_\ell$ will be precisely the map of Corollary \ref{mapping cone proj} obtained by resolving the crossing $\sigma_{j_{\ell+1}}$.  The maps $f_k$ are just compositions of such maps, as in \cite{Roz}.
\item The maps $F_\ell$ will be constructed via iterating Corollary \ref{mapping cone proj} over a careful choice of crossings to resolve.
\item The function $k=z(\ell)$ will be based upon how far along the infinite braid $\B$ we must look before we can ``see'' the braid $\T^{nk}$ sitting within $\B_\ell$.
\item The estimates on $|F_\ell|_h$ will be based upon Corollary \ref{mapping cone proj} together with careful use of Equations (\ref{R1 lower}) and (\ref{R2}).  Similar arguments will estimate $|g_\ell|_h$ to guarantee that $\{\mathbf{B_\ell},g_\ell\}$ was indeed Cauchy.
\end{enumerate}

\subsection{The Details}
\label{sec_details}

Fix the number of strands $n$.  We begin with a semi-infinite, right-handed, complete braid $\B$ and set out to prove Theorem \ref{inf braid gives jw} via Proposition \ref{comparing systems} using the list of the overview.  The points 1-4 of the overview require no further explanation.  We begin with points 5 and 6, that is, the construction of the map
\[F_\ell: \KC(\B_\ell) \rightarrow \KC(\T^k)\]
where $k=z(\ell)$ must be determined.

Given the braid $\B_\ell$, we start at the top of the braid (beginning of the braid word) and seek the first occurrence of generator $\sigma_1$.  From that point we go downward and find the first occurrence of $\sigma_2$, and so forth until we reach $\sigma_{n-1}$.  In this way we have found crossings within $\B_\ell$ that would, in the absence of the crossings we ``skipped'', give a single copy of $\T^1$.  We connect these crossings with a dashed line going rightward then downward as in Figure \ref{creating diagonals}, and we call such a set of crossings a \emph{diagonal}.  The crossings involved are called \emph{diagonal crossings}.  Having found such a diagonal within $\B_\ell$, we work our way back \emph{up} the braid $\B_\ell$ in the same way going from the diagonal $\sigma_{n-1}$ to the previous (not necessarily diagonal) $\sigma_{n-2}$ and so forth until we reach another $\sigma_1$ (if there were no skipped crossings, we are now back at the $\sigma_1$ we started with).  We begin the second diagonal from the first $\sigma_1$ that is \emph{below} this $\sigma_1$ we found at the end of our upward journey.  In this way we find disjoint diagonals with as few ``skipped'' crossings between them as possible.  See Figure \ref{creating diagonals} for clarification.

\begin{figure}
\begin{center}
\includegraphics[scale=.42]{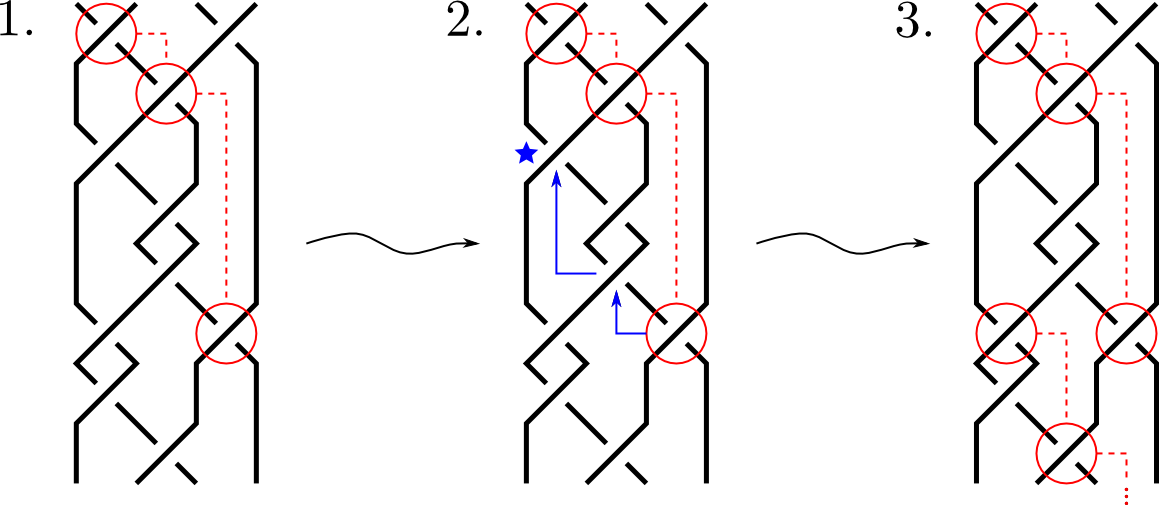}
\end{center}
\caption{An illustration of finding diagonals within some $\B_\ell$.  In step 1, we find the first diagonal illustrated in red.  In step 2, we work our way back up from the diagonal $\sigma_{n-1}$ as in the blue arrows until we arrive at the $\sigma_1$ marked by a blue star.  In step 3, we begin forming the second diagonal starting from the first $\sigma_1$ below the starred crossing from step 2.}
\label{creating diagonals}
\end{figure}

Let $y(\ell)$ denote the number of diagonals that can be completed within $\B_\ell$ in this way.  The function $z(\ell)$ determining the destination of the map $F_\ell$ is
\begin{equation}
\label{z(ell) eqn}
z(\ell):=\left\lfloor\frac{y(\ell)}{n}\right\rfloor
\end{equation}
where $\lfloor\cdot\rfloor$ denotes the integer floor function.  Thus $z(\ell)$ gives the number of \emph{full} twists that can be seen within $\B_\ell$.  The map $F_\ell$ is then the composition of maps coming from Corollary \ref{mapping cone proj} where we are resolving all non-diagonal crossings in $\B_\ell$.  Note that the order in which we resolve the crossings is irrelevant, and in fact the map $F_\ell$ can be viewed as a projection from the single mapping cone of the direct sum of the Khovanov maps assigned to each non-diagonal crossing.  However in this paper we shall consider $F_\ell$ as a large composition starting from resolving the bottom-most (non-diagonal) crossing.  From this consideration it should be clear that the maps $F_\ell$ commute with the maps $f_k$ and $g_\ell$ of the two systems $\KC(\T^k)$ and $\KC(\B_\ell)$, which are also just maps based on resolving bottom-most crossings.

We now move on to point 7 from the overview.  We wish to estimate $|F_\ell|_h$.  Viewing $F_\ell$ as a composition of projections from crossing resolutions as above, we estimate the homological order of the cone of the $i^{\text{th}}$ such projection with the help of Corollary \ref{mapping cone proj}.  That is, we view $\KC(\B_\ell)$ as an iterated mapping cone
\begin{align*}
\h^{n^-}&\q^{-N}KC(\B_\ell) =\\
&\left( \left( \left( \cdots \rightarrow \KCq{1}(\mathbf{T}_3)\right) \rightarrow \KCq{1}(\mathbf{T}_2)\right) \rightarrow \KCq{1}(\mathbf{T}_1) \right)
\end{align*}
and we consider the minimum homological order of
\[\KChq{1}{1}(\mathbf{T}_i)\]
where $\mathbf{T}_i$ is a tangle that is obtained from $\B_\ell$ by resolving the first $i-1$ non-diagonal crossings (starting from the bottom of the braid) as 0-resolutions, and then resolving the $i^{\text{th}}$ non-diagonal crossing as a 1-resolution.  Iterating Lemma \ref{main mapping cone} over all of the remaining non-diagonal crossings, we can see $\KChq{1}{1}(\mathbf{T}_i)$ as a large multi-cone as illustrated in Figure \ref{KC multicone ex}.

\begin{figure}
\begin{center}
\includegraphics[scale=.2]{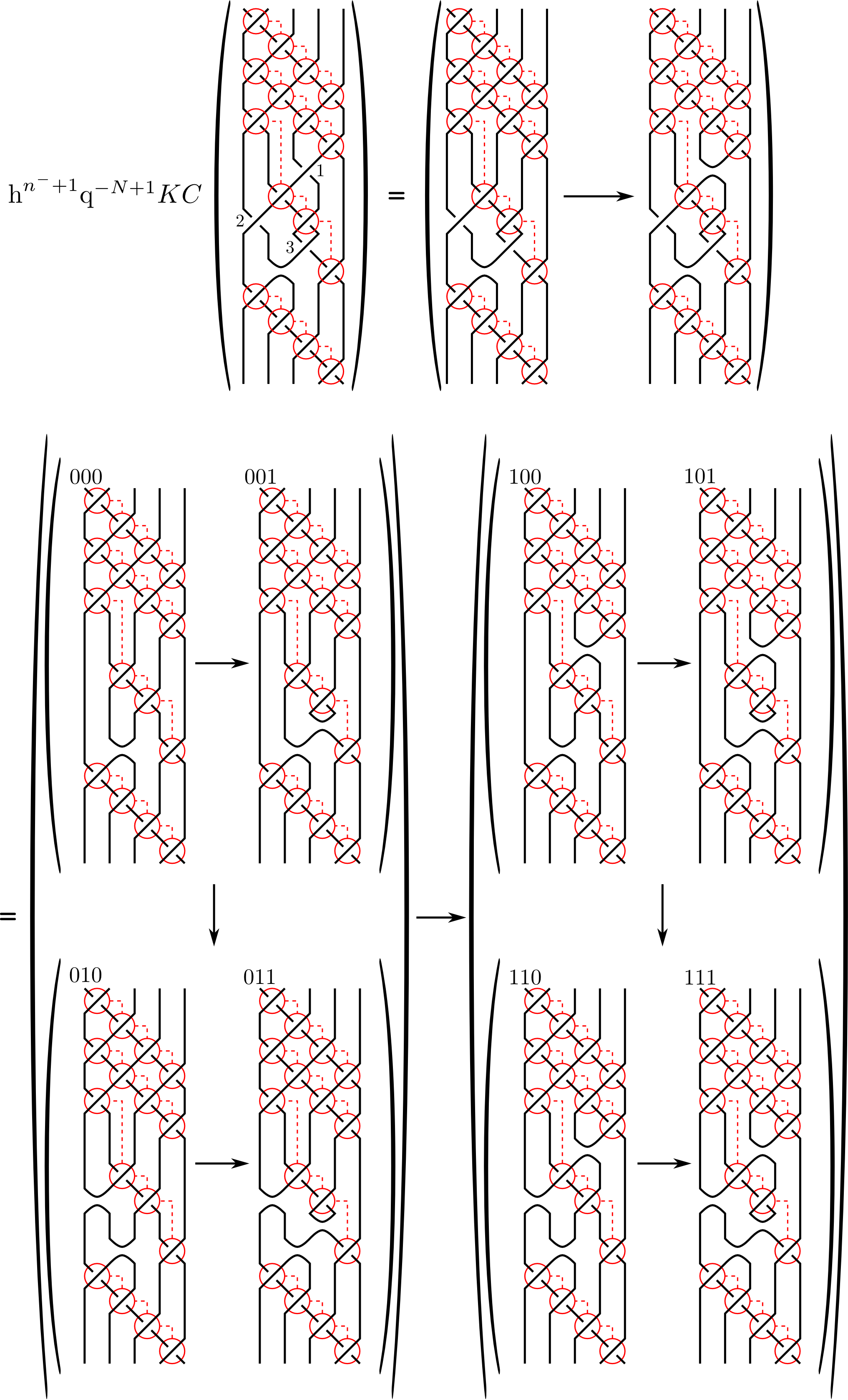}
\end{center}
\caption{A multicone presentation for an example $\KChq{1}{1}(\mathbf{T}_i)$.  The shifts and the $KC$ notation are suppressed.  Each term carries a shift of $\h^{n^-+1}\q^{-N+1+r}$, where $r$ is the sum of the three resolution numbers above each diagram indicating which resolution was taken for each of the three non-diagonal crossings (as numbered in the starting diagram).}
\label{KC multicone ex}
\end{figure}

Note that every diagram within the large multi-cone for $\KChq{1}{1}(\mathbf{T}_i)$ is made up of diagonal crossings and possible turnbacks from 1-resolutions (\hres) between the diagonals.  Indeed we are guaranteed at least the one turnback pair (\hres) already present within $\mathbf{T}_i$, but there may be many more.  Now we turn to the key lemma that produces the required estimate on $|F_\ell|_h$.

\begin{lemma}
\label{key reid lemma}
Let $\D$ be any $(n,n)$ tangle diagram involving precisely $y$ diagonals of crossings, no other crossings, and at least one pair of turnbacks between the diagonals (see the diagrams in Figure \ref{KC multicone ex}).  Then $\D$ can be simplified to a new diagram $\D'$ via Reidemeister 3, Reidemeister 2, and negative Reidemeister 1 moves.  During this process, all of the Reidemeister 2 and negative Reidemeister 1 moves remove crossings, and the total number of such moves is at least $y$.
\end{lemma}
\begin{proof}
We view the $y$ diagonals as partitioning the diagram $\D$ into $y+1$ \emph{zones}, and we call such a zone \emph{empty} if there are no turnbacks (\hres) within it.  By assumption there is at least one non-empty zone.  We start from the topmost non-empty zone, and choose the `bottommost' such pair in this zone (ie, the last $\sigma_{j_m}$ within the given zone where a 1-resolution occurred).  The lower turnback can then be passed through the diagonals below it one at a time via Reidemeister II moves (Figure \ref{downward R2}) and negative Reidemeister I moves (first step of Figure \ref{downward R1}) until the turnback reaches the next non-empty zone.  Note that, following a Reidemeister I move into an empty zone, multiple moves are required to pass through the next diagonal (also illustrated in Figure \ref{downward R1}).  Nevertheless, it is clear that during this process, the number of such Reidemeister moves will be at least the same as the number of diagonals passed through.

\begin{figure}
\begin{center}
\includegraphics[scale=.3]{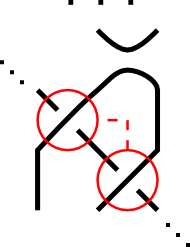}
\end{center}
\caption{Pulling a turnback downward through a diagonal via Reidemeister II.}
\label{downward R2}
\end{figure}
\begin{figure}
\begin{center}
\includegraphics[scale=.3]{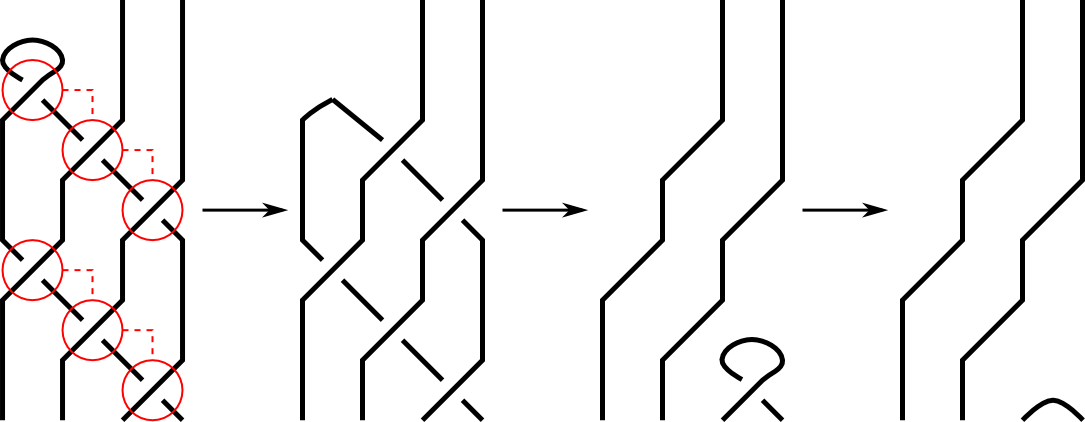}
\end{center}
\caption{An example of pulling a turnback downward through two diagonals via negative Reidemeister I and Reidemeister II moves.}
\label{downward R1}
\end{figure}

Having now reached the second non-empty zone, we find the bottommost turnback within this zone and continue the process until the final zone is reached.  This accounts for passing through all the diagonals below the topmost non-empty zone.  Finally, we return to that starting zone and choose the `topmost' turnback within that zone (ie the first $\sigma_{j_m}$ within that zone where a 1-resolution occurred) and pass this turnback through all of the diagonals above it.  If the first move required is a Reidemeister II move (ie the two strands connected by the turnback are adjacent on the defining torus of the twist), this process will be the same as the downward one.  If the first move is a (negative) Reidemeister I move, this process may require some Reidemeister III moves as illustrated in Figure \ref{upward Reid moves}.  However it is clear that there will still be at least as many Reidemeister II and negative Reidemeister I moves as there are diagonals, and thus the total number of such moves is at least $y$ as desired.
\begin{figure}
\begin{center}
\includegraphics[scale=.3]{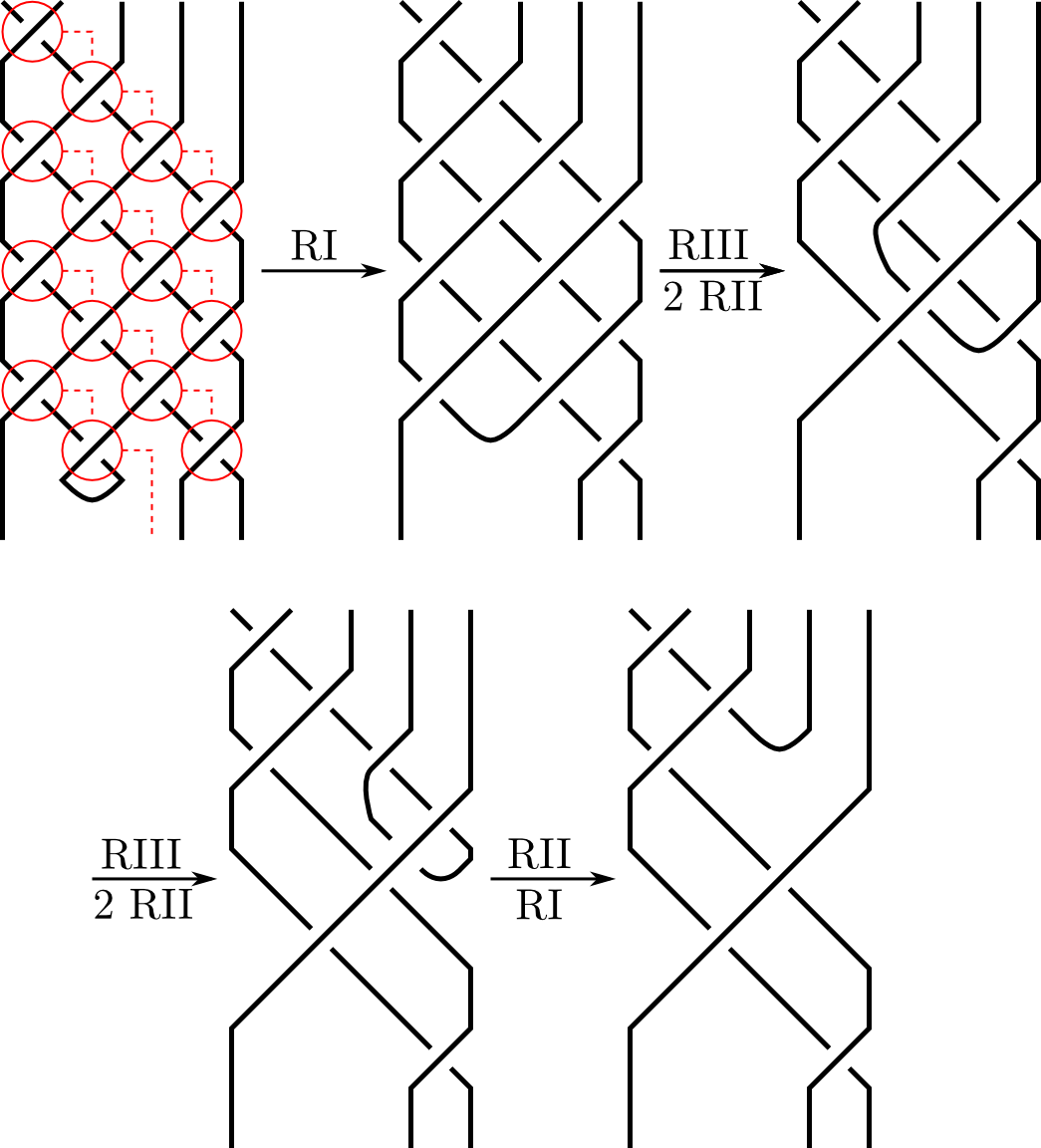}
\end{center}
\caption{An example of pulling a turnback upward through diagonals.  Each step indicates passing through one diagonal, so that the total number of negative Reidemeister I and Reidemeister II moves is clearly at least the number of such diagonals.  
}
\label{upward Reid moves}
\end{figure}

More conceptually, a sequence of diagonals with empty zones between them corresponds to a torus braid.  A topmost turnback below this (or bottommost turnback above this) corresponds to connecting two strands of the torus braid.  The simpler cases above correspond to these two strands being adjacent, while the more complex case of Figure \ref{upward Reid moves} corresponds to connecting two non-adjacent strands.  Either way, the turnback can be pulled up (or down) through the center of the torus as in Figure \ref{torus braid pull up}.  Passing through diagonals corresponds to passing by other strands, which must eliminate crossings, thus necessitating at least one Reidemesiter I or II move.  The Reidemeister I moves must be negative because they are undoing right-handed twisting.
\begin{figure}
\begin{center}
\includegraphics[scale=.3]{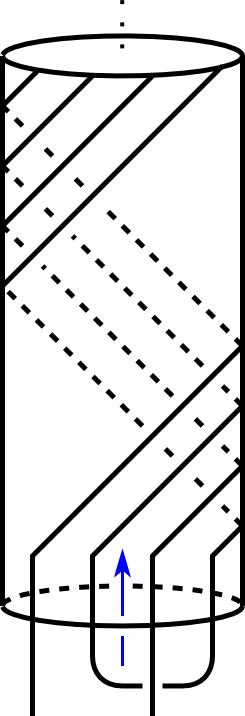}
\end{center}
\caption{A topmost turnback entering a set of diagonals corresponds to a turnback being pulled through the center of a torus braid as shown here.  The blue arrow indicates the direction of the pulling.}
\label{torus braid pull up}
\end{figure}
\end{proof}

\begin{corollary}
\label{key shifting estimate}
Every term $\KChq{1}{1+r}(\D)$ in the multicone expansion of any $\KChq{1}{1}(\mathbf{T}_i)$ (see Figure \ref{KC multicone ex}) is chain homotopy equivalent to a complex of the form $\KChq{1+s_h}{1+r+s_q}(\D')$ where $s_h$ and $s_q$ are homological and $q$-degree shifts depending on the expansion term, and $r$ is the number of 1-resolutions taken to arrive at $\D$ from $\mathbf{T_i}$.  Moreover, for any term in the expansion, $s_q\geq s_h\geq y$.
\end{corollary}
\begin{proof}
The shifts come from Equations (\ref{R1 lower}) and (\ref{R2}).
\end{proof}

As an example of Corollary \ref{key shifting estimate}, consider the (111)-entry from Figure \ref{KC multicone ex}.  We illustrate the process of Lemma \ref{key reid lemma}, keeping track of the shifts, for this entry in Figure \ref{111 simp ex}.  In this case we get $s_h=4=y$, while $s_q=5$.  As an illustration of the case where Reidemeister III moves are also required, we show the process for the (001)-entry in Figure \ref{001 simp ex} where $s_h=8>y$ and $s_q=10$.  Notice that further simplifications are possible in the first case, indicating that our given bounds will rarely be sharp.

\begin{figure}
\begin{center}
\includegraphics[scale=.28]{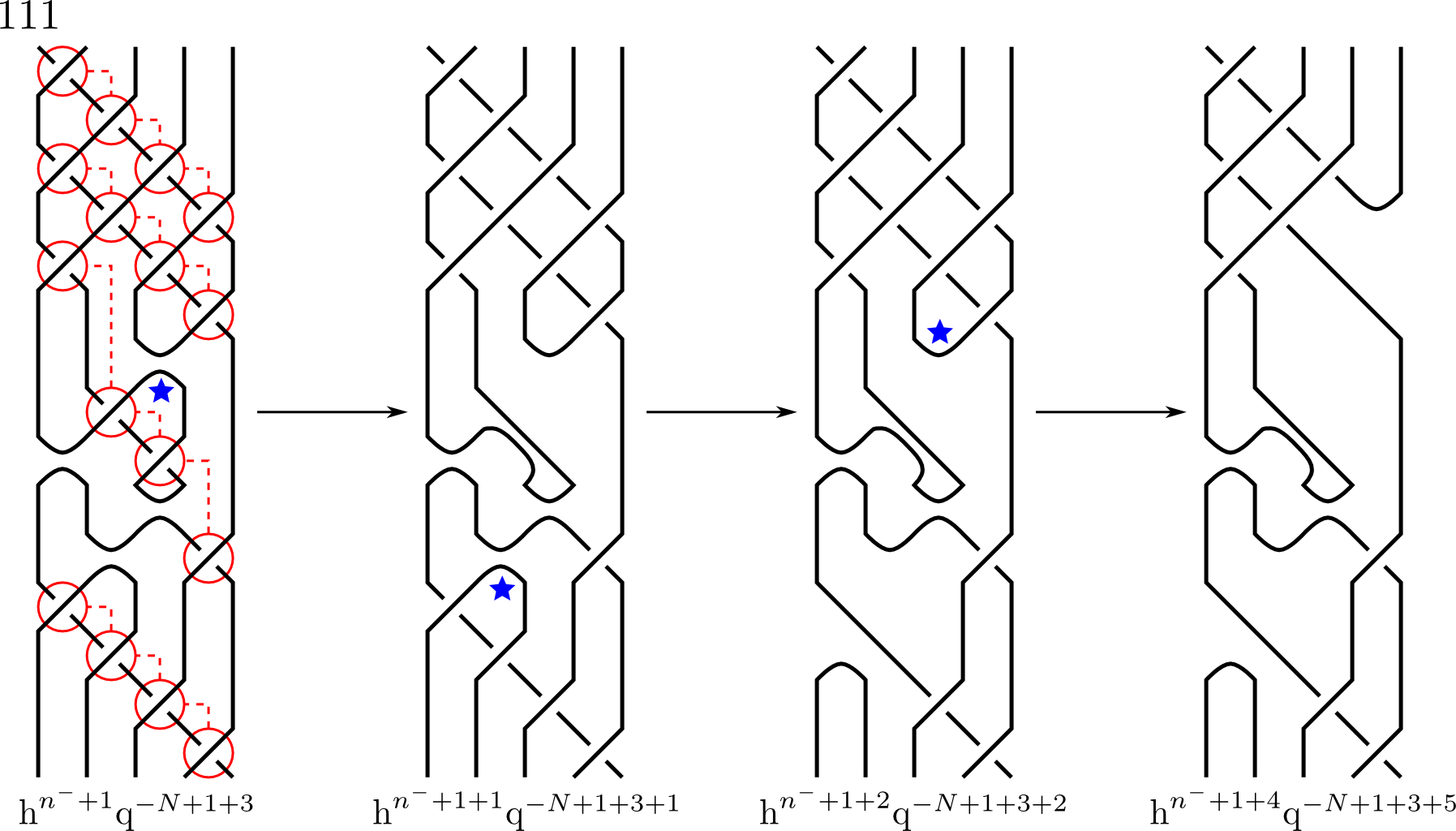}
\end{center}
\caption{The process of Lemma \ref{key reid lemma} shown for the (111)-entry from Figure \ref{KC multicone ex}, illustrating the degree shifts of Corollary \ref{key shifting estimate}.  The turnback that is about to be `pulled' is indicated by a blue star.}
\label{111 simp ex}
\end{figure}
\begin{figure}
\begin{center}
\includegraphics[scale=.28]{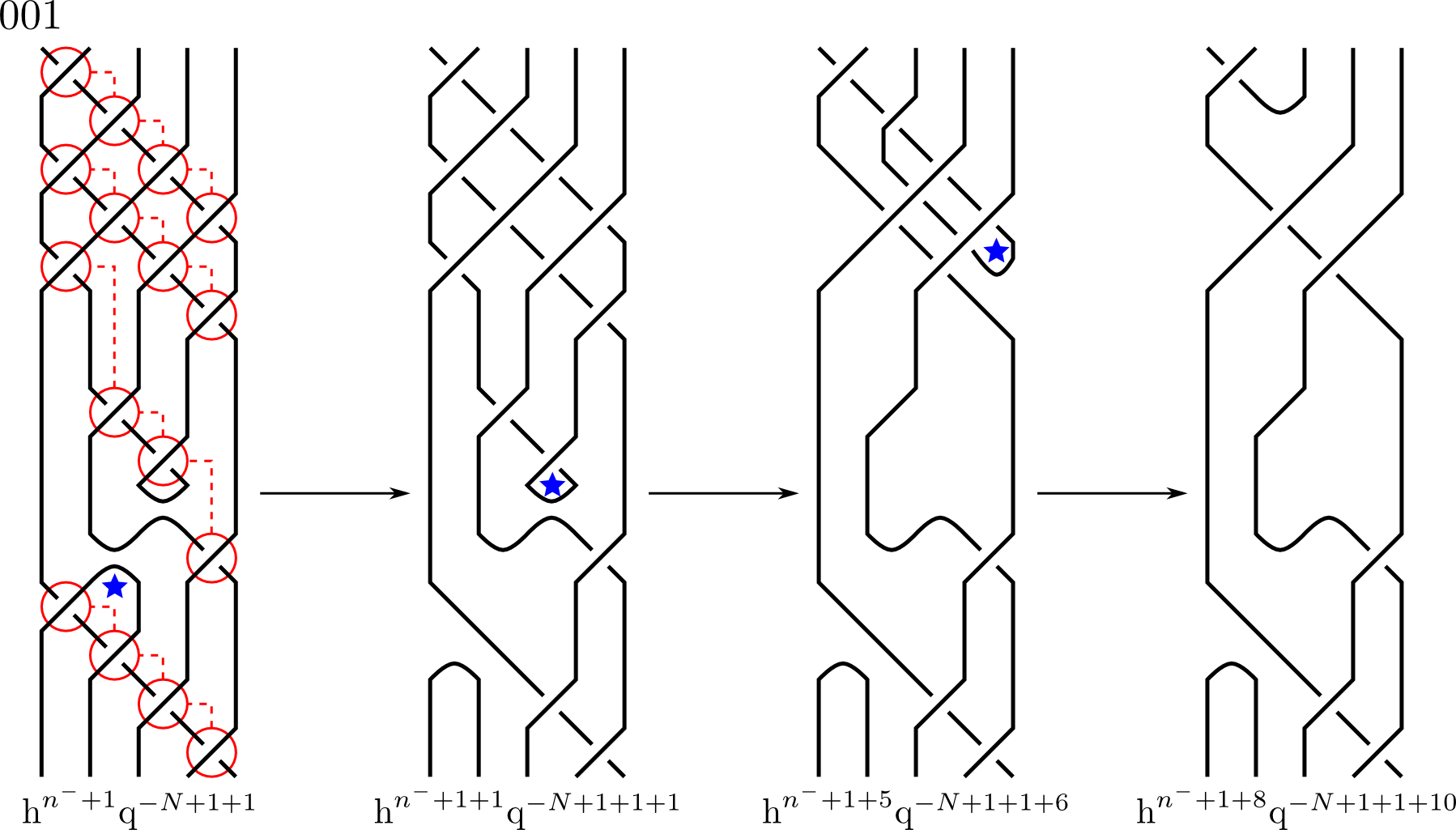}
\end{center}
\caption{The process of Lemma \ref{key reid lemma} shown for the (001)-entry from Figure \ref{KC multicone ex}, illustrating the degree shifts of Corollary \ref{key shifting estimate}.  The turnback that is about to be `pulled' is indicated by a blue star.}
\label{001 simp ex}
\end{figure}

\begin{proof}[Proof of Theorem \ref{inf braid gives jw}]
We build a commuting diagram as in Figure \ref{comparing systems fig} using the listed points of the overview.  The construction of $F_\ell$ as a composition of mapping cone projections $p$ ensures $|F_\ell|_h$ is at least as large as the minimum $|p|_h$ amongst all such $p$.  As described above, this is precisely the minimum homological order amongst all the $\KChq{1}{1}(\mathbf{T}_i)$ via Corollary \ref{mapping cone proj}.  Corollary \ref{key shifting estimate} guarantees that the minimal homological degree of any term in the multicone expansion of such a complex (and thus for the entire complex) is at least $y$, the number of diagonals found in $\B_\ell$.  Thus we have
\[|F_\ell|_h \geq y.\]
The assumption that the semi-infinite braid $\B$ is complete ensures that $y\rightarrow\infty$ as $\ell\rightarrow\infty$.  The mapping cones of the maps $g_\ell$ also involve diagrams with turnbacks, so that a similar (and simpler) argument also ensures $|g_\ell|_h\rightarrow\infty$ as $\ell\rightarrow\infty$, verifying that this system is Cauchy and has a limit.  Thus we may use Proposition \ref{comparing systems} to conclude the proof.
\end{proof}

\section{Proving Theorem \ref{inf braid gives jw X}}
\label{sec_proof X}

The proof of Theorem \ref{inf braid gives jw X} is a very simple generalization of the proof of Theorem \ref{inf braid gives jw} to the setting of the L-S-K homotopy type.  In short, we build a diagram similar to Figure \ref{comparing systems fig} out of homotopy types instead of chain complexes.  Then instead of tracking the homological order below which the maps are chain homotopy equivalences, we track the $q$-degree below which the maps are stable homotopy equivalences.  Corollary \ref{key shifting estimate} ensures that this maximal $q$-degree of equivalence goes to infinity as the sequence of maps goes to infinity.

To begin with, we recall some of the key properties of the homotopy type $\X(L)$ of an oriented link $L$.  The first will allow us to focus on a single $q$-degree at a time.
\begin{proposition}[Theorem 1 in \cite{LS}]
\label{qdeg decomp}
The L-S-K homotopy type $\X(L)$ of an oriented link $L$ decomposes as a wedge sum over $q$-degree
\begin{equation}
\label{qdeg decomp eqn}
\X(L) = \bigvee_{j\in\Z} \X^j(L)
\end{equation}
where for each $j\in\Z$, the cochain complex of $\X^j(L)$ matches the Khovanov chain complex in $q$-degree $j$
\begin{equation}
\label{qdeg def}
C^i(\X^j(L)) = KC^{i,j}(L).
\end{equation}
\end{proposition}
Note that for any link $L$ the number of non-empty $q$-degrees is finite, and so the wedge sum of Proposition \ref{qdeg decomp} is actually finite.

The following property lifts Lemma \ref{main mapping cone}.
\begin{proposition}[Theorem 2 in \cite{LS}]
\label{main cofib seq}
Let $L$ be an oriented link, with a specified crossing \crossing.  Let $L_0$ denote the same link with the crossing replaced by its 0-resolution \vres , and let $L_1$ denote the same with the 1-resolution \hres .  Then for each $q$-degree $j\in\Z$, the corresponding homotopy types fit into a cofibration sequence
\begin{equation}
\label{main cofib seq eqn}
\Sigma^{n^-}\X^{j-N}(L_0) \hookrightarrow \Sigma^{n^-}\X^{j-N}(L) \twoheadrightarrow \Sigma^{n^-+1}\X^{j-N+1}(L_1)
\end{equation}
where $\Sigma$ denotes the suspension operator, and the notations $n^-$ and $N$ follow the conventions set out in Definition \ref{norm notation}.
\end{proposition}
The cofibration sequence (\ref{main cofib seq eqn}) can be combined over the wedge sum (\ref{qdeg decomp eqn}) to give a cofibration sequence over the full homotopy type that we write as
\begin{equation}
\label{main cofib seq eqn full}
\Sigma^{n^-}\q^{-N}\X(L_0) \hookrightarrow \Sigma^{n^-}\q^{-N}\X(L) \twoheadrightarrow \Sigma^{n^-+1}\q^{-N+1}\X(L_1)
\end{equation}
where, abusing notation slightly, we again use the $\q$ operator to indicate a shifting of the $q$-degrees assigned to each wedge summand of $\X(L)$.

\begin{corollary}
\label{cofib equivalence}
If for some $q$-degree $j\in\Z$ we have $KC^{j-N+1}(L_1)$ homologically trivial, then the inclusion map
\begin{equation}
\label{cofib inclusion}
\Sigma^{n^-}\X^{j-N}(L_0) \hookrightarrow \Sigma^{n^-}\X^{j-N}(L)
\end{equation}
is a stable homotopy equivalence.
\end{corollary}
\begin{proof}
The cofibration sequence (\ref{main cofib seq eqn}) gives rise to a long exact sequence on homology, and the assumption ensures that the map above gives isomorphisms on all homology.  Therefore by Whitehead's theorem it is a stable homotopy equivalence (there is no notion of a $\pi_1$ obstruction in the stable homotopy category).
\end{proof}

We now generalize the definitions needed to discuss stable limits of infinite sequences of L-S-K homotopy types, as in Definition \ref{inverse system def} and Theorem \ref{Cauchy has limit}.  We do not need a notion of sequences of homotopy types being `Cauchy', but we do need some notion of stability.
\begin{definition}
\label{max qdeg of equivalence}
A map between L-S-K homotopy types $f:\XX(L)\rightarrow\XX(L')$ is called \emph{$q$-homogeneous} if $f$ preserves normalized $q$-degrees between wedge summands.  That is,
\begin{equation}
\label{q homog def}
f=\vee_{j\in\Z}f^j
\end{equation}
for $q$-preserving maps
\[f^j:\X^{j-N}(L) \rightarrow \X^{j-N}(L).\]
In this case, we let $|f|_q$ denote the maximal $q$-degree $d$ for which $f^j$ is a stable homotopy equivalence for all $j\leq d$.
\end{definition}
It is clear from the definitions that the maps of Equation (\ref{main cofib seq eqn full}) are $q$-homogeneous.
\begin{definition}
\label{q stable def}
An infinite sequence of $q$-homogeneous maps
\begin{equation}
\label{seq of X maps}
\XX(L_0) \xrightarrow{f_0} \XX(L_1) \xrightarrow {f_1} \XX(L_2) \xrightarrow \cdots
\end{equation}
will be called a \emph{direct $q$-system} of L-S-K homotopy types, denoted $\{\X(L_k),f_k\}$.  Such a system
is called \emph{$q$-stable} if $|f_k|_q\rightarrow\infty$ as $k\rightarrow\infty$.
\end{definition}

\begin{theorem}
\label{qstable has good limit}
A $q$-stable direct $q$-system $\{\X(L_k),f_k\}$ has homotopy colimit 
\begin{align}
\label{hocolim eqn}
\X(L_\infty):=&\hocolim\left( \XX(L_0) \xrightarrow{f_0} \XX(L_1) \xrightarrow {f_1} \XX(L_2) \xrightarrow \cdots \right)\\
&= \bigvee_{j\in\Z} \X^j(L_\infty)
\end{align}
where
\[\X^j(L_\infty):=\hocolim\left( \X^{j-N}(L_0) \xrightarrow{f_0} \X^{j-N}(L_1) \xrightarrow {f_1} \X^{j-N}(L_2) \xrightarrow \cdots \right)\]
and for each $j$, there exists lower bound $k_j$ such that
\[\X^j(L_\infty) \simeq \X^{j-N}(L_k) \hspace{.1in}\forall k\geq k_j.\]
\end{theorem}
\begin{proof}
This is clear from the properties of a homotopy colimit after unwinding the definitions.
\end{proof}
\begin{remark}
\label{direct not inverse}
Notice that the maps of the sequence (\ref{seq of X maps}) go in the opposite direction as those of Equation (\ref{inverse system eqn}) considered earlier for chain complexes.  This is to be expected, since the inverse system of Equation (\ref{inverse system eqn}) should be recovered by the singular cochain functor $C^*$ which is \emph{contravariant}.  Similarly, our limits here are homotopy \emph{co}limits, as opposed to the inverse limits considered in the previous section.
\end{remark}

With these ideas in place, we can state and prove the homotopy version of Proposition \ref{comparing systems} providing the diagram corresponding to Figure \ref{comparing systems fig}.
\begin{proposition}
\label{comparing X systems}
Suppose $\{\X(L_k),f_k\}$ and $\{\X(M_\ell),g_\ell\}$ are $q$-stable direct $q$-systems with homotopy colimits $\X(L_\infty)$ and $\X(M_\infty)$ respectively, as in Equation (\ref{hocolim eqn}).  Suppose there are $q$-homogeneous maps
\[F_\ell : \XX(L_{z(\ell)}) \rightarrow \XX(M_\ell)\]
($z(\ell)$ is an increasing function of $\ell$, not necessarily strict) forming a commuting diagram with the system maps $f_k$ and $g_\ell$.  If $|F_\ell|_q\rightarrow\infty$ as $\ell\rightarrow\infty$, then we have \[\X(L_\infty)\simeq\X(M_\infty).\]
\end{proposition}
\begin{proof}
The proof is very similar to that of Proposition \ref{comparing systems}.  See Figure \ref{comparing X systems fig}.  The properties of homotopy colimits provide the existence of $q$-homogeneous maps $\tilde{f}_k:\XX(L_k) \rightarrow \X(L_\infty)$ and $\tilde{g}_\ell:\XX(M_\ell) \rightarrow \X(M_\infty)$ as well as the map $F_\infty:\X(L_\infty)\rightarrow\X(M_\infty)$ which must commute with all of the other maps.  Fixing some $q$-degree $j$, Theorem \ref{qstable has good limit} and the assumption on the maps $F_\ell$ guarantee that the wedge summand maps $\tilde{f}^j_k$, $\tilde{g}^j_\ell$ and $F^j_\ell$ all become stable homotopy equivalences once $k$ and $\ell$ are large enough.  Thus $F_\infty$ must also provide a stable homotopy equivalence $F^j_\infty: \X^j(L_\infty)\xrightarrow{\simeq}\X^j(M_\infty)$.  This happens for all $j$, so in fact $F_\infty$ is the desired ($q$-homogeneous) stable homotopy equivalence.
\end{proof}
\begin{figure}
\begin{center}
\includegraphics[scale=.5]{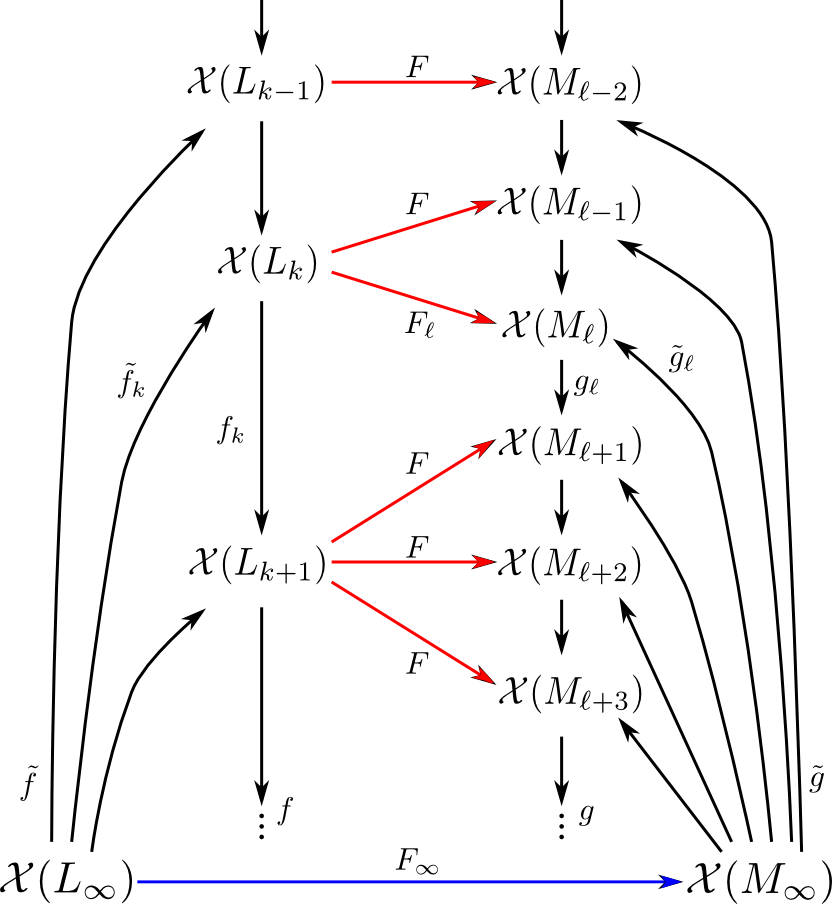}
\end{center}
\caption{The diagram for Proposition \ref{comparing X systems}, omitting the normalization shifts $\Sigma^{n^-}\q^{-N}$ on each term.  Compare to Figure \ref{comparing systems fig}.  Given the two $q$-stable systems $\{\X(L_k),f_k\}$ and $\{\X(M_\ell),g_\ell\}$, the homotopy colimits $\X(L_\infty)$ and $\X(M_\infty)$ come with maps $\tilde{f}$ and $\tilde{g}$.  If we find the $q$-homogeneous maps $F$ (red) and show that $|F_\ell|_q\rightarrow\infty$ as $\ell\rightarrow\infty$, then the map $F_\infty$ on the colimits (blue) is a stable homotopy equivalence.}
\label{comparing X systems fig}
\end{figure}

\begin{proof}[Proof of Theorem \ref{inf braid gives jw X}]
Given a specified closure $\overline{\B}$ of a complete semi-infinite right-handed braid $\B$ on $n$ strands, we build the diagram of Figure \ref{comparing X systems fig} in a manner completely analogous the building of the diagram of Figure \ref{comparing systems fig}.
\begin{itemize}
\item The links $L_k$ are the corresponding closures of the full twists $L_k:=\overline{\T^{nk}}$.
\item The maps $f_k$ are compositions of the cofibration maps of Equation (\ref{cofib inclusion}) coming from resolving the crossings of the last full twist in $\overline{\T^{n(k+1)}}$ as 0-resolutions (see \cite{MW2}).
\item The links $M_\ell$ are the corresponding closures of the partial braids $\B_\ell$, that is, $M_\ell:=\overline{\B_\ell}$.
\item The maps $g_\ell$ are the inclusion maps of Equation (\ref{main cofib seq eqn full}) coming from resolving the last crossing of $\B_{\ell+1}$ as a 0-resolution.
\item The maps $F_\ell$ are compositions of inclusions coming from resolving non-diagonal crossings as 0-resolutions precisely as in the proof of Theorem \ref{inf braid gives jw}.
\end{itemize}
See Figure \ref{closures of infinite braids} for the notion of corresponding closures of braids.  It is shown in \cite{MW2} that the direct system $\{\X(\overline{\T^{nk}}),f_k\}$ is $q$-stable, with homotopy colimit $\X(\overline{\T^\infty})$ satisfying many properties similar to closures of the Jones-Wenzl projectors $\PP_n$.  The proof that $|F_\ell|_q$ and $|g_\ell|_q$ go to infinity with $\ell$ is analogous to the similar statement about $|F_\ell|_h$ and $|g_\ell|_h$ in the proof of Theorem \ref{inf braid gives jw}.  In short, we use Corollary \ref{cofib equivalence} to change the question to one of homological triviality of $KC^{j-N+1}(\overline{\mathbf{T_i}})$ for closures of braids $\mathbf{T}_i$ involving diagonals and turnbacks, as before.  The estimate of Corollary \ref{key shifting estimate} still holds, but now we are concerned with the minimum $q$-value (rather than minimum homological value) of a complex of the form $\KChq{1+s_h}{1+r+s_q}(\overline{\D'})$ (where again $\D'$ came from a partial resolution $\D$ of $\mathbf{T}_i$ by pulling turnbacks through diagonals).  For this purpose we define
\begin{equation}
\label{num circ}
\sharp^\circ(L(0)):=\text{the number of circles in the all-zero resolution of the link }L.
\end{equation}
Then the minimum $q$-degree for generators of the complex $\KChq{1+s_h}{1+r+s_q}(\overline{\D'})$ is precisely
\begin{equation}
\label{minq of D'}
\min_q\left(\KChq{1+s_h}{1+r+s_q}(\overline{\D'})\right) = 1+r+s_q-\sharp^\circ(\overline{\D'}(0))
\end{equation}
since each circle in the all-zero resolution of the link can contribute a generator $v_-$ with $q$-degree -1.

Now the number of circles in a resolution is bounded above by the number of local maxima (\inlinegfx{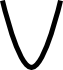}) present in the diagram.  The number of such maxima in the all-zero resolution of any $\overline{\D'}$ is comprised of two parts, those within the tangle $\D'$ and those without (so those due to the specified closure).  The second category will contribute some constant $c$ that is independent of the tangle $\D'$, and indeed independent of the infinite braid $\B$ at all.  The first category will be bounded above by the number of 1-resolutions that were taken to arrive at $\D$ from $\B_\ell$ (note that the process of pulling turnbacks through diagonals does not create maxima).  But this number is precisely $1+r$.  Thus we have
\begin{align*}
\min_q\left(\KChq{1+s_h}{1+r+s_q}(\overline{\D'})\right) &= 1+r+s_q-\sharp^\circ(\overline{\D'}(0))\\
&\geq 1+r+s_q-(1+r+c)\\
&\geq y-c
\end{align*}
which certainly goes to infinity as $y$ does.  The assumption of completeness ensures $y\rightarrow\infty$ as $\ell\rightarrow\infty$, and so we have $|F_\ell|_q\rightarrow\infty$ as $\ell\rightarrow\infty$.  As in the proof of Theorem \ref{inf braid gives jw}, the argument for $|g_\ell|_q$ is a simpler version of this, and so we are done.
\end{proof}

\section{More General Infinite Braids}
\label{sec_general inf braids}

In this section we collect a handful of corollaries of Theorems \ref{inf braid gives jw} and \ref{inf braid gives jw X} for dealing with other types of infinite braids.

\begin{corollary}
\label{fin neg crossings}
Let $\B$ be a complete semi-infinite braid containing only finitely many left-handed crossings $\left(\lcrossing\right)$.  Then $KC(\B)$ is chain homotopy equivalent to a shifted categorified Jones-Wenzl projector $\h^a\q^b\PP_n$, and similarly for the L-S-K homotopy types $\X(\overline{\B})\simeq \Sigma^a\q^b\X(\overline{\T^\infty})$.
\end{corollary}
\begin{proof}
If there are only finitely many left-handed crossings, we can view $\B$ as the product of the finite partial braid $\B_m$ which contains all of these crossings, and the infinite braid $\B'$ which consists of the rest of $\B$.  Then the result follows from the similar properties of $\PP_n$ (see \cite{Roz}) and $\X(\overline{\T^\infty})$ (see \cite{MW2}).  The shifts $a$ and $b$ will depend on the orientations of the crossings in the finite $\B_m$.
\end{proof}

To give the most general possible statement, we start with a definition.
\begin{definition}
\label{tangle with inf braids}
A \emph{tangle involving semi-infinite braids} is a tangle diagram $\mathbf{Z}$ where any finite number of interior discs $\D^i$ containing only the identity tangles $I_{n_i}$ are formally replaced by complete semi-infinite right-handed braids $\B^i$ (see Figure \ref{tangle with inf braids fig}). 
\end{definition}

\begin{figure}
\begin{center}
\includegraphics[scale=.4]{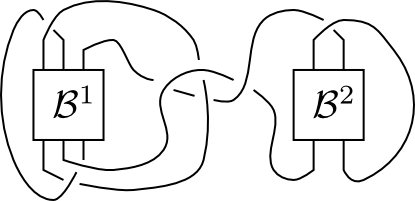}
\end{center}
\caption{An example of a closed tangle involving semi-infinite braids $\B^1$ and $\B^2$.  As long as both are right-handed and complete, the resulting Khovanov chain complex and homotopy type will match those of the same diagram with infinite twists in place of the $\B^1$ and $\B^2$.}
\label{tangle with inf braids fig}
\end{figure}

\begin{theorem}
\label{tangle with inf braids gives jw tangle}
For any tangle $\mathbf{Z}$ involving finitely many complete semi-infinite right-handed braids $\B^i$ on $n_i$ strands, the Khovanov chain complex $KC(\mathbf{Z})$ (defined in a limiting sense analogous to that of $KC(\B)$ in Theorem \ref{inf braid gives jw}) is chain homotopy equivalent to the Khovanov complex of the same tangle where the $\B^i$ have been replaced with the corresponding $\PP_{n_i}$.  Similarly, if the tangle $\mathbf{Z}$ is closed, then $\X(\mathbf{Z})$ is stably homotopy equivalent to the same tangle where the $\B^i$ have been replaced with the corresponding infinite twist $\T_{n_i}^\infty$.
\end{theorem}
\begin{proof}
This is immediate for the projectors $\PP_{n_i}$ which are defined via braids that allow for stitching; the corresponding statement for homotopy types of tangles involving infinite twists was proved in \cite{MW2}, which allows for this generalization.
\end{proof}

Theorem \ref{tangle with inf braids gives jw tangle} allows us to consider many sorts of infinite (right-handed) braids by breaking them up into complete semi-infinite (right-handed) braids.  For instance, a non-complete semi-infinite braid is equivalent to a tangle involving a finite braid and two or more complete semi-infinite braids below it (see Figure \ref{noncomplete}).  As another example, a bi-infinite braid $\B=\cdots \sigma_{j_{-2}}\sigma_{j_{-1}}\sigma_{j_0}\sigma_{j_1}\sigma_{j_2}\cdots$ can be viewed as the composition of two semi-infinite braids $\B=\B^-\cdot\B^+$ (see Figure \ref{bi-infinite}).  In this way we see that many different notions of infinite braids have limiting Khovanov complex (and L-S-K homotopy type, if closed) made up of combinations of Jones-Wenzl projectors (or homotopy types involving closures of infinite twists).  Choices of how to arrange the diagrams (for instance, where to begin the semi-infinite complete braid in Figure \ref{noncomplete}) lead to normalization shifts within the resulting complex or homotopy type similar to those in \cite{MW2} and \cite{Roz} (as in Corollary \ref{fin neg crossings}).

\begin{figure}
\begin{center}
\includegraphics[scale=.3]{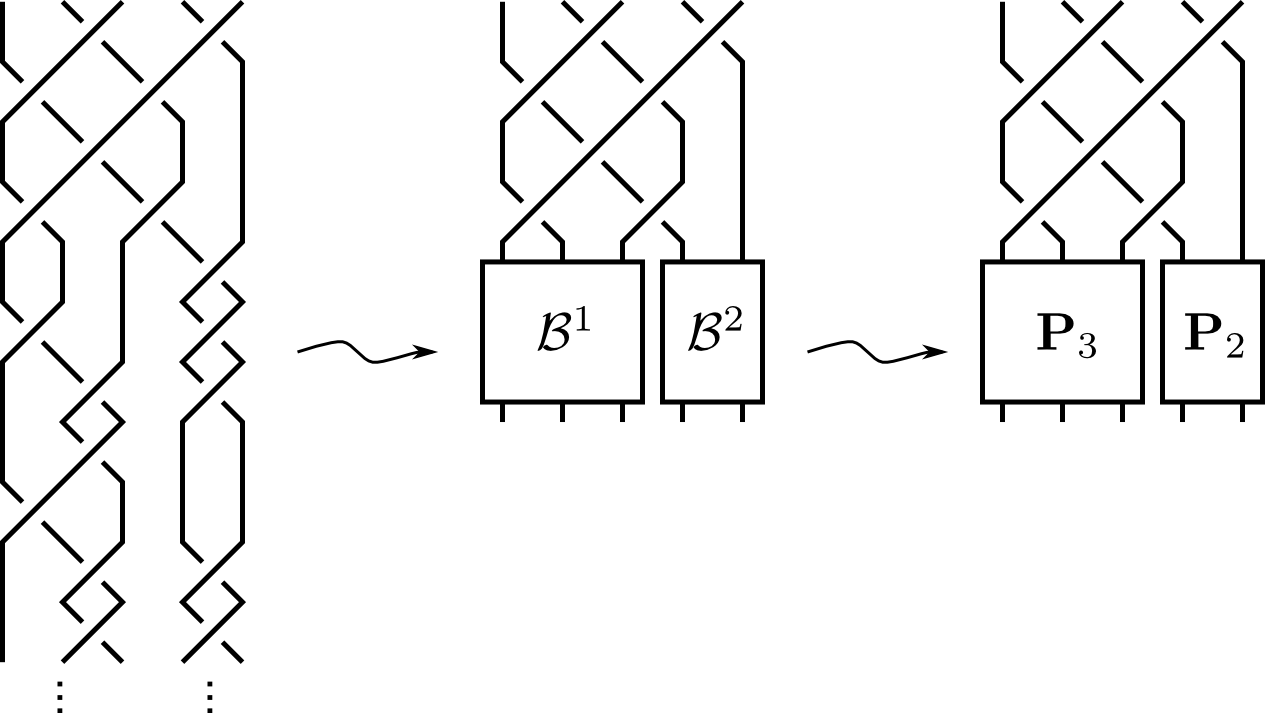}
\end{center}
\caption{Viewing a non-complete infinite braid as a combination of two complete ones, which limit to their respective $\PP_{n_i}$.}
\label{noncomplete}
\end{figure}
\begin{figure}
\begin{center}
\includegraphics[scale=.3]{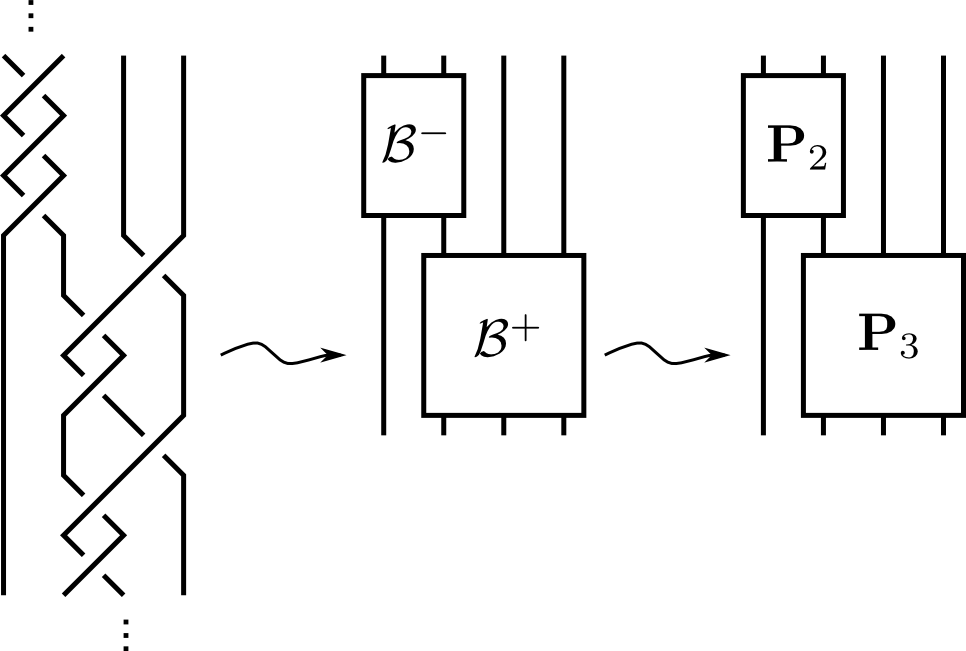}
\end{center}
\caption{Viewing a bi-infinite braid as a combination of two semi-infinite ones, which limit to their respective $\PP_{n_i}$.}
\label{bi-infinite}
\end{figure}

\newpage

\bibliographystyle{alpha}

\bibliography{Infinite_Braids}

\end{document}